\documentclass[12pt]{article}
\usepackage{amsmath, amssymb,latexsym}
\usepackage{color}
\title{Resolvent estimates for elliptic quadratic differential operators}
\author{Michael Hitrik\\Department of Mathematics \\University of California \\ Los Angeles
\\CA 90095-1555, USA\\\small hitrik@math.ucla.edu \and
Johannes Sj\"ostrand\\IMB, Universit\'e de Bourgogne\\9, Av. A. Savary, BP 47870\\FR--21078 Dijon, France \\and UMR 5584 CNRS
\\\small johannes.sjostrand@u-bourgogne.fr \and Joe Viola \\ Mathematical Sciences \\  Lund University
\\Box 118 \\ S-221 00 Lund, Sweden \\\small jviola@maths.lth.se}
\date{}

\def\wrtext#1{\relax\ifmmode{\leavevmode\hbox{#1}}\else{#1}\fi}
\def\abs#1{\left|#1\right|}
\def\begeq{\begin{equation}}
\def\endeq{\end{equation}}

\def\neigh{neighborhood}

\def\Re{{\rm Re\,}}
\def\Im{{\rm Im\,}}

\def\part#1{\frac{\partial}{\partial #1}}

\def\norm#1{||\,#1\,||}

\newcommand{\real}{\mbox{\bf R}}
\newcommand{\comp}{\mbox{\bf C}}

\newcommand{\nat}{\mbox{\bf N}}

\renewcommand{\Re}{\mbox{\rm Re\,}}
\renewcommand{\Im}{\mbox{\rm Im\,}}

\renewcommand{\exp}{\mbox{\rm exp\,}}

\newtheorem{dref}{Definition}[section]
\newtheorem{lemma}[dref]{Lemma}
\newtheorem{theo}[dref]{Theorem}
\newtheorem{prop}[dref]{Proposition}

\newenvironment{proof}{\vspace{.3cm}\noindent{{\em Proof:}}}{\hfill$\Box$}

\begin{document}
\maketitle

\vspace*{1cm}
\noindent
{\bf Abstract}: Sharp resolvent bounds for non-selfadjoint semiclassical ellip\-tic quad\-ratic diffe\-ren\-tial ope\-rators are
estab\-lished, in the interior of the range of the associated quadratic symbol.



\tableofcontents
\section{Introduction and statement of result}
\setcounter{equation}{0}
It is well known that the spectrum of a non-selfadjoint operator does not control its resolvent, and that the latter may become very
large even far from the spectrum. Understanding the behavior of the norm of the resolvent of a given non-selfadjoint operator is therefore a
natural and basic problem, which has recently received a considerable attention, in particular, within the circle of questions around the notion
of the pseudospectrum,~\cite{EmTr}. Some general upper bounds on resolvents are provided by the abstract operator theory, and restricting
the attention to the setting of semiclassical pseudodifferential operators on $\real^n$, relevant for this note, let us recall a rough statement
of such bounds, following~\cite{DeSjZw},~\cite{Markus},~\cite{Viola}. Assume that $P=p^w(x,hD_x)$ is the semiclassical Weyl quantization on
$\real^n$ of a complex-valued smooth symbol $p$ belonging to a suitable symbol class and satisfying an ellipticity condition at infinity,
guaranteeing that the spectrum of $P$ is discrete in a small \neigh{} of the origin. Then the norm of the $L^2$--resolvent of $P$ is
bounded from above by a quantity of the form ${\cal O}(1) \exp\left({\cal O}(1) h^{-n}\right)$, provided that $z\in {\rm neigh}(0,\comp)$ is
not too close to the spectrum of $P$. On the other hand, the available lower bounds on the resolvent of $P$, coming from the
pseudospectral considerations, are typically of the form $C_N^{-1}h^{-N}$, $N\in \nat$, or $(1/C)e^{1/(Ch)}$, provided that
$p$ enjoys some analyticity properties,~\cite{DeSjZw}. There appears to be therefore a substantial gap between the available
upper and lower bounds on the resolvent, especially when $n\geq 2$. The purpose of this note is to address the issue of
bridging this gap in the particular case of an elliptic quadratic semiclassical differential operator on $\real^n$, and to
establish a sharp upper bound on the norm of its resolvent.

\bigskip
\noindent
Let $q$ be a complex-valued quadratic form,
\begeq
\label{eq0.1}
q: \real^n_x\times \real^n_{\xi}\rightarrow \comp,\quad (x,\xi)\mapsto q(x,\xi).
\endeq
We shall assume throughout the following discussion that the quadratic form $q$ is elliptic on $\real^{2n}$, in the sense that $q(X)=0$,
$X\in \real^{2n}$, precisely when $X=0$. In this case, according to Lemma 3.1 of~\cite{Sjostrand74}, if $n>1$, then there exists
$\lambda\in \comp$, $\lambda\neq 0$, such that $\Re (\lambda q)$ is positive definite. In the case when $n=1$, the same conclusion
holds, provided that the range of $q$ on $\real^2$ is not all of $\comp$,~\cite{Sjostrand74},~\cite{Hi2004}, which is going to be
assumed in what follows. After a multiplication of $q$ by $\lambda$, we may and will assume henceforth that $\lambda=1$, so that
\begeq
\label{eq0.2}
\Re q>0.
\endeq
It follows that the range $\Sigma(q) = q(\real^{2n})$ of $q$ on $\real^{2n}$ is a closed angular sector with a vertex at zero,
contained in the union of $\{0\}$ and the open right half-plane.

\medskip
\noindent
Associated to the quadratic form $q$ is the semiclassical Weyl quantization $q^w(x,hD_x)$, $0<h\leq 1$, which we shall view as a
closed densely defined operator on $L^2(\real^n)$, equipped with the domain $\{u\in L^2(\real^n); q^w(x,hD_x)u\in L^2(\real^n)\}$.
The spectrum of $q^w(x,hD_x)$ is discrete, and following~\cite{Sjostrand74}, we shall now recall its explicit description.
To that end, let us introduce the Hamilton map $F$ of $q$,
$$
F: \comp^{2n}\rightarrow \comp^{2n},
$$
defined by the identity,
\begeq
\label{eq0.3}
q(X,Y) = \sigma(X,FY),\quad X,Y\in \comp^{2n}.
\endeq
Here the left hand side is the polarization of $q$, viewed as a symmetric bilinear form on $\comp^{2n}$, and $\sigma$ is the complex symplectic
form on $\comp^{2n}$. We notice that the Hamilton map $F$ is skew-symmetric with respect to $\sigma$, and furthermore,
\begeq
\label{eq0.31}
F Y = \frac{1}{2} H_q (Y),
\endeq
where $H_q = q'_{\xi}\cdot \partial_x - q'_x\cdot \partial_{\xi}$ is the Hamilton field of $q$.

\medskip
\noindent
The ellipticity condition (\ref{eq0.2}) implies that the spectrum of the Hamilton map $F$ avoids the real axis, and in general we know from
Section 21.5 of~\cite{H_book} that if $\lambda$ is an eigenvalue of $F$, then so is $-\lambda$, and the algebraic multiplicities agree.
Let $\lambda_1,\ldots, \lambda_n$ be the eigenvalues of $F$, counted according to their multiplicity, such that $\lambda_j/i\in \Sigma(q)$,
$j=1,\ldots, n$. Then the spectrum of the operator $q^w(x,hD_x)$ is given by the eigenvalues of the form
\begeq
\label{eq0.4}
h \sum_{j=1}^n \frac{\lambda_j}{i} \left(2\nu_{j,\ell}+1\right),\quad \nu_{j,\ell}\in \nat\cup\{0\}.
\endeq
We notice that ${\rm Spec}(q^w(x,hD_x))\subset \Sigma(q)$, and from~\cite{Pravda_XEDP} we also know that
$$
{\rm Spec}(q^w(x,hD_x))\cap \partial \Sigma(q)=\emptyset,
$$
provided that the operator $q^w(x,hD_x)$ is not normal.

\medskip
\noindent
The following is the main result of this work.

\begin{theo}
Let $q: \real^n_x\times \real^n_{\xi}\rightarrow \comp$ be a quadratic form such that $\Re q$ is positive definite.
Let $\Omega \subset \subset \comp$. There exists $h_0>0$ and for every $C>0$ there exists $A>0$ such that
\begeq
\label{est1}
\norm{\left(q^w(x,hD_x)-z\right)^{-1}}_{{\cal L}(L^2({\bf R}^n), L^2({\bf R}^n))} \leq A\, \exp\left(A h^{-1}\right),
\endeq
for all $h\in (0,h_0]$, and all $z\in \Omega$, with ${\rm dist}\, (z,{\rm Spec}(q^w(x,hD_x)))\geq 1/C$. Furthermore, for all $C>0$, $L\geq 1$, there
exists $A>0$ such that for $h\in (0,h_0]$, we have
\begeq
\label{est2}
\norm{\left(q^w(x,hD_x)-z\right)^{-1}}_{{\cal L}(L^2({\bf R}^n), L^2({\bf R}^n))} \leq A\, \exp\left(A h^{-1}\log \frac{1}{h}\right),
\endeq
if the spectral parameter $z\in \Omega$ is such that
$$
{\rm dist}\, (z,{\rm Spec}(q^w(x,hD_x)))\geq h^L/C.
$$
\end{theo}

\bigskip
\noindent
{\it Remark}. Assume that the elliptic quadratic form $q$, with $\Re q>0$, is such that the Poisson bracket $\{\Re q,\Im q\}$ does not
vanish identically, and let $z\in \Sigma(q)^{o}$, $z\notin {\rm Spec}(q^w(x,hD_x))$. Here $\Sigma(q)^o$ is the interior of $\Sigma(q)$.
Then it follows from the results of~\cite{DeSjZw} and~\cite{PravdaDuke} that we have the following lower bound for
$\left(q^w(x,hD_x)-z\right)^{-1}$, as $h\rightarrow 0$,
$$
\norm{\left(q^w(x,hD_x)-z\right)^{-1}}_{{\cal L}(L^2({\bf R}^n), L^2({\bf R}^n))} \geq \frac{1}{C_0} e^{1/(C_0h)}\,\quad C_0>0.
$$
It follows that the upper bound (\ref{est1}) is of the right order of magnitude, when $z\in \Sigma(q)^{o}\cap \Omega$, $\abs{z}\sim 1$,
avoids a closed cone $\subset \Sigma(q)\cup \{0\}$, containing the spectrum of $q^w(x,hD_x)$.

\bigskip
\noindent
{\it Remark}. In Section 4 below, we shall give a simple example of an elliptic quadratic operator on $\real^2$, for which the associated
Hamilton map has a non-vanishing  nilpotent part in its Jordan decomposition, and whose resolvent exhibits the superexponential growth given
by the right hand side of (\ref{est2}), in the region of the complex spectral plane where $\abs{z}\sim 1$,
${\rm dist}(z,{\rm Spec}(q^w(x,hD_x)))\sim h$. On the other hand, sharper resolvent estimates can be obtained when the Hamilton
map $F$ of $q$ is diagonalizable. In this case, in Section 4 we shall see that the bound (\ref{est2}) improves to the following, when
$z\in \Omega$ and $h\in (0,h_0]$,
\begeq
\label{est3}
\norm{\left(q^w(x,hD_x)-z\right)^{-1}}_{{\cal L}(L^2({\bf R}^n), L^2({\bf R}^n))} \leq \frac{A e^{A/h}}{{\rm dist}\,(z, {\rm Spec}(q^w(x,hD_x)))}.
\endeq

\bigskip
\noindent
{\it Remark}. Let $z_0\in {\rm Spec}(q^w(x,hD_x))\cap \Omega$ and let
$$
\Pi_{z_0} = \frac{1}{2\pi i}\int_{\partial D} \left(z-q^w(x,hD_x)\right)^{-1}\,dz
$$
be the spectral projection of $q^w(x,hD_x)$, associated to the eigenvalue $z_0$. Here $D \subset \Omega$ is a small open disc centered at $z_0$,
such that the closure $\overline{D}$ avoids the set ${\rm Spec}(q^w(x,hD_x))\backslash \{z_0\}$, and $\partial D$ is its positively oriented boundary.
Assume for simplicity that the quadratic form $q$ is such that its Hamilton map is diagonalizable. Then it follows from (\ref{est3}) that
$$
\Pi_{z_0}  = {\cal O}(1) \exp\left({\cal O}(1) h^{-1}\right): L^2 (\real^n)\rightarrow L^2(\real^n).
$$

\bigskip
\noindent
In the context of elliptic quadratic differential operators in dimension one, resolvent bounds have been studied, in particular,
in~\cite{Boulton},~\cite{Da2},~\cite{DaKu}. We should also mention the general resolvent estimates
of~\cite{DeSjZw},~\cite{Sj2010}, valid for $h$--pseudodifferential operators, when the spectral parameter is close to the boundary of
the range of the corresponding symbol.

\medskip
\noindent
The plan of this note is as follows. In Section 2, we make an essentially well-known reduction of our problem to the setting of a
quadratic differential operator, acting in a Bargmann space of holomorphic functions, convenient for the subsequent analysis.
Section 3 is devoted to suitable a priori elliptic estimates, valid for holomorphic functions vanishing to a high, $h$--dependent, order at
the origin. The proof of Theorem 1.1 is completed in Section 4 by some elementary considerations in the space of holomorphic polynomials
on $\comp^n$, of degree not exceeding ${\cal O}(h^{-1})$.

\bigskip
\noindent
{\bf Acknowledgements}. The second author has benefitted from support of the Agence Nationale de la Recherche under the references JC05-52556 and
ANR-08-BLAN-0228-01, as well as a grant FABER of the Conseil r\'egional de Bourgogne.

\section{The normal form reduction}
\setcounter{equation}{0}
We shall be concerned here with a quadratic form $q: T^*\real^n \rightarrow \comp$, such that $\Re q$ is positive definite. Let
$F$ be the Hamilton map of $q$, introduced in (\ref{eq0.3}). When $\lambda \in {\rm Spec}(F)$, we let
\begeq
\label{eq4}
V_{\lambda} = {\rm Ker}((F-\lambda)^{2n})\subset T^*\comp^{n}
\endeq
be the generalized eigenspace belonging to the eigenvalue $\lambda$. The symplectic form $\sigma$ is then non-degenerate viewed as a bilinear form
on $V_{\lambda}\times V_{-\lambda}$.

\medskip
\noindent
Let us introduce the stable outgoing manifold for the Hamilton flow of the quadratic form $i^{-1}q$, given by
\begeq
\label{eq5}
\Lambda^+  := \bigoplus_{{\rm Im}\, \lambda>0} V_{\lambda}\subset T^*\comp^n.
\endeq
It is then true that $\Lambda^+$ is a complex Lagrangian plane such that $q$ vanishes along $\Lambda^+$, and Proposition 3.3 of~\cite{Sjostrand74}
states that the complex Lagrangian $\Lambda^+$ is strictly positive in the sense that
\begeq
\label{eq6}
\frac{1}{i} \sigma(X,\overline{X})>0,\quad 0\neq X\in \Lambda^+.
\endeq
We also define
\begeq
\label{eq7}
\Lambda^-  = \bigoplus_{{\rm Im}\, \lambda<0} V_{\lambda}\subset T^*\comp^n,
\endeq
which is a complex Lagrangian plane such that $q$ vanishes along $\Lambda^-$, and from the arguments of~\cite{Sjostrand74} we also know that
$\Lambda^-$ is strictly negative in the sense that
\begeq
\label{eq8}
\frac{1}{i} \sigma(X,\overline{X})<0,\quad 0\neq X\in \Lambda^-.
\endeq

\bigskip
\noindent
The complex Lagrangians $\Lambda^+$ and $\Lambda^-$ are transversal, and following~\cite{HeSj84},~\cite{Sj86}, we would like to
implement a reduction of the quadratic form $q$ to a normal form by applying a linear complex canonical transformation which reduces $\Lambda^+$ to
$\{(x,\xi)\in \comp^{2n};\, \xi=0\}$ and $\Lambda^-$ to $\{(x,\xi)\in \comp^{2n};\, x=0\}$. We shall then be
able to implement the canonical transformation in question by an FBI--Bargmann transform. Let us first simplify $q$ by means of a suitable real
linear canonical transformation. When doing so, we observe that the fact that the Lagrangian $\Lambda^-$ is strictly
negative implies that it is of the form
$$
\eta = A_-y,\quad y\in \comp^n,
$$
where the complex symmetric $n\times n$ matrix $A_-$ is such that $\Im A_-<0$. Here $(y,\eta)$ are the standard canonical coordinates on
$T^*\real^n_y$, that we extend to the complexification $T^*\comp^n_y$. Using the real linear canonical transformation
$(y,\eta)\mapsto (y, \eta - (\Re A_-)y)$, we reduce $\Lambda^-$ to the form $\eta = i\Im A_-y$, and by a diagonalization of $\Im A_-$, we
obtain the standard form $\eta = -iy$. After this real
linear symplectic change of coordinates, and the conjugation of the semiclassical Weyl quantization $q^w(x,hD_x)$ of $q$ by means of the
corresponding unitary metaplectic operator, we may assume that $\Lambda^-$ is of the form
\begeq
\label{eq9}
\eta = -iy,\quad y\in \comp^n,
\endeq
while the positivity property of the complex Lagrangian $\Lambda^+$ is unaffected, so that, in the new real symplectic coordinates,
extended to the complexification, $\Lambda^+$ is of the form
\begeq
\label{eq10}
\eta = A_+ y,\quad \Im A_+>0.
\endeq

\bigskip
\noindent
Let
\begeq
\label{eq11}
B = B_+ = (1-iA_+)^{-1}A_+,
\endeq
and notice that the matrix $B$ is symmetric. Let us introduce the following FBI--Bargmann transform,
\begeq
\label{eq12}
Tu(x) = C h^{-3n/4} \int e^{i\varphi(x,y)/h} u(y)\,dy,\quad x\in \comp^n,\quad C>0,
\endeq
where
\begeq
\label{eq12.5}
\varphi(x,y)  = \frac{i}{2}(x-y)^2 - \frac{1}{2}(Bx,x).
\endeq
The associated complex linear canonical transformation on $\comp^{2n}$,
\begeq
\label{eq13}
\kappa_T: (y,-\varphi'_y(x,y))\mapsto (x,\varphi'_x(x,y))
\endeq
is of the form
\begeq
\label{eq14}
\kappa_T: (y,\eta)\mapsto (x,\xi)=(y-i\eta, \eta+iB\eta-By),
\endeq
and we see that the image of $\Lambda_-:\eta=-iy$ under $\kappa_T$ is the fiber $\{(x,\xi)\in \comp^{2n};\, x=0\}$, while
$\kappa_T(\Lambda^+)$ is given by the equation $\{(x,\xi)\in \comp^{2n}; \xi=0\}$.

\medskip
\noindent
We know from~\cite{Sj95} that for a suitable choice of $C>0$ in (\ref{eq12}), the map $T$ is unitary,
\begeq
\label{eq15}
T: L^2(\real^n)\rightarrow H_{\Phi_0}(\comp^n),
\endeq
where
$$
H_{\Phi_0}(\comp^n) = {\rm Hol}\,(\comp^n)\cap L^2(\comp^n; e^{-2\Phi_0/h}L(dx)),
$$
and $\Phi_0$ is a strictly plurisubharmonic quadratic form on $\comp^n$, given by
\begeq
\label{eq16}
\Phi_0(x) = {\rm sup}\,_{y\in {\bf R}^n} \left(-\Im \varphi(x,y)\right) = \frac{1}{2} \left(\left(\Im x\right)^2 + \Im (Bx,x)\right).
\endeq
From~\cite{Sj95}, we recall also that the canonical transformation $\kappa_T$ in (\ref{eq13}) maps $\real^{2n}$ bijectively onto
\begeq
\label{eq17}
\Lambda_{\Phi_0}:=\left\{\left(x,\frac{2}{i}\frac{\partial \Phi_0}{\partial x}(x)\right); x\in \comp^n\right\}.
\endeq
As explained in Chapter 11 of~\cite{Sj82}, the strict positivity of $\kappa_T(\Lambda^+)=\{(x,\xi)\in \comp^{2n}; \xi=0\}$ with
respect to $\Lambda_{\Phi_0}$ implies that the quadratic weight function $\Phi_0$ is strictly convex, so that
\begeq
\label{eq18}
\Phi_0(x) \sim \abs{x}^2,\quad x\in \comp^n.
\endeq

\medskip
\noindent
We have next the exact Egorov property,~\cite{Sj95},
\begeq
\label{eq19}
T q^w(y,hD_y)u = \widetilde{q}^w(x,hD_x) Tu,\quad u\in {\cal S}(\real^n),
\endeq
where $\widetilde{q}$ is a quadratic form on $\comp^{2n}$ given by $\widetilde{q} = q\circ \kappa_T^{-1}$. It follows therefore that
\begeq
\label{eq20}
\widetilde{q}(x,\xi) = Mx \cdot \xi,
\endeq
where $M$ is a complex $n\times n$ matrix. We have
$$
H_{\widetilde{q}} = Mx \cdot \partial_x - M^{t}\xi\cdot \partial_{\xi},
$$
and using (\ref{eq0.31}), we conclude that with the agreement of algebraic multiplicities, the following holds,
\begeq
\label{eq21}
{\rm Spec}(M) = {\rm Spec}(2F)\cap \{\Im \lambda>0\}.
\endeq
The problem of estimating the norm of the resolvent of $q^w(x,hD_x)$ on $L^2(\real^n)$ is therefore equivalent to controlling the norm of the resolvent
of the quadratic operator $\widetilde{q}^w(x,hD_x)$, acting in the space $H_{\Phi_0}(\comp^n)$, where the quadratic weight $\Phi_0$ enjoys the property
(\ref{eq18}).

\medskip
\noindent
In what follows, it will be convenient to reduce the matrix $M$ in (\ref{eq20}) to its Jordan normal form. To this end, let us notice that we can
implement this reduction by considering a complex canonical transformation of the form
\begeq
\label{eq22}
\kappa_{C}: \comp^{2n}\ni (x,\xi)\mapsto (C^{-1}x, C^{t}\xi)\in \comp^{2n},
\endeq
where $C$ is a suitable invertible complex $n\times n$ matrix. On the operator level, associated to the transformation in (\ref{eq22}), we have the
operator $u(x)\mapsto \abs{{\rm det}\, C} u(Cx)$, which maps the space $H_{\Phi_0}(\comp^n)$ unitarily
onto the space $H_{\Phi_1}(\comp^n)$, where $\Phi_1(x)=\Phi_0(Cx)$ is a strictly plurisubharmonic quadratic weight such that
$\kappa_C(\Lambda_{\Phi_0}) = \Lambda_{\Phi_1}$. We notice that the property
\begeq
\label{eq23}
\Phi_1(x) \sim \abs{x}^2,\quad x\in \comp^n,
\endeq
remains valid.

\medskip
\noindent
We summarize the discussion pursued in this section, in the following result.
\begin{prop}
Let $q: \real^n_x\times \real^n_{\xi}\rightarrow \comp$ be a quadratic form, with $\Re q>0$. The operator
$$
q^w(x,hD_x): L^2(\real^n)\rightarrow L^2(\real^n),
$$
equipped with the domain
$$
{\cal D}(q^w(x,hD_x)) = \{u\in L^2(\real^n); \left(x^2+(hD_x)\right)^2 u\in L^2(\real^n)\},
$$
is unitarily equivalent to the quadratic operator,
$$
\widetilde{q}^w(x,hD_x): H_{\Phi_1}(\comp^n)\rightarrow H_{\Phi_1}(\comp^n),
$$
with the domain
$$
{\cal D}(\widetilde{q}^w(x,hD_x)) = \{u\in H_{\Phi_1}(\comp^n); (1+\abs{x}^2)u\in L^2_{\Phi_1}(\comp^n)\}.
$$
Here
$$
\widetilde{q}(x,\xi) = Mx\cdot \xi,
$$
where $M$ is a complex $n\times n$ block--diagonal matrix, each block being a Jordan one. Furthermore, the eigenvalues of $M$ are precisely those of
$2F$ in the upper half-plane, and the quadratic weight function $\Phi_1(x)$ satisfies,
$$
\Phi_1(x)\sim \abs{x}^2,\quad x\in \comp^n.
$$
We have the ellipticity property,
\begeq
\label{eq24}
\Re \widetilde{q}\left(x,\frac{2}{i}\frac{\partial \Phi_1}{\partial x}(x)\right)\sim \abs{x}^2,\quad x\in \comp^n.
\endeq
\end{prop}

\medskip
\noindent
{\it Remark}. The normal form reduction described in Proposition 2.1 is close to the corresponding discussion of Section 3
in~\cite{Sjostrand74}. Here, for future computations, it will be convenient for us to work in the Bargmann space $H_{\Phi_1}(\comp^n)$.

\section{An elliptic estimate}
\setcounter{equation}{0}
Following the reduction of Proposition 2.1, here we shall be concerned with the quadratic operator $\widetilde{q}^w(x,hD_x)$, acting on
$H_{\Phi_1}(\comp^n)$.
The purpose of this section is to establish a suitable a priori estimate for holomorphic functions, vanishing to a high,
$h$-dependent, order at the origin, instrumental in the proof of Theorem 1.1. The starting point is the following observation, which
comes directly from Lemma 4.5 in~\cite{GeSj}, and whose proof we give for the convenience of the reader only.

\begin{lemma}
Let $u\in {\rm Hol}(\comp^n)$ and assume that $\partial^{\alpha}u(0)=0$, $\abs{\alpha}<N$. Assume that $0<C_0<C_1<\infty$. Then
\begeq
\label{eq3.1}
\norm{u}_{L^{\infty}(B(0, C_0))} \leq \left(N \frac{C_1}{C_1-C_0}\right)\left(\frac{C_0}{C_1}\right)^N \norm{u}_{L^{\infty}(B(0, C_1))}.
\endeq
Here $B(0,C_j)=\{x\in \comp^n;\, \abs{x}\leq C_j\}$, $j=0,1$.
\end{lemma}

\begin{proof}
By Taylor's formula, we have
$$
u(x)  = \int_0^1 \frac{(1-t)^{N-1}}{(N-1)!} \left(\frac{d}{dt}\right)^N u(tx)\,dt.
$$
We may assume that $\abs{x}=C_0$, and apply Cauchy's inequalities, so that
$$
\abs{\left(\frac{d}{dt}\right)^N u(tx)} \leq \frac{C_0^N N!}{(C_1-C_0 t)^N}\norm{u}_{L^{\infty}(B(0,C_1))}.
$$
It suffices therefore to remark that the expression
$$
N \int_0^1 \frac{(1-t)^{N-1}}{(C_1/C_0-t)^N}\,dt
$$
does not exceed
$$
\left(\frac{N}{\frac{C_1}{C_0}-1}\right) \left(\frac{C_0}{C_1}\right)^{N-1}.
$$
\end{proof}

\bigskip
\noindent
Let $K>0$ be fixed and assume that $u\in H_{\Phi_1}(\comp^n)$ is such that $\partial^{\alpha}u(0)=0$, when $\abs{\alpha}<N$. Using Lemma 3.1, we write
\begin{multline}
\norm{u}^2_{H_{\Phi_1}(B(0,K))} \leq \norm{u}^2_{L^2(B(0,K))} \\
\leq {\cal O}_K(1)\norm{u}^2_{L^{\infty}(B(0,K))} \leq {\cal O}_K(1) N^2 e^{-2N} \norm{u}^2_{L^{\infty}(B(0,Ke))} \\
\leq {\cal O}_K(1) N^2 e^{-2N}\norm{u}^2_{L^2(B(0,(K+1)e))}
\leq {\cal O}_K(1) N^2 e^{-2N} e^{\frac{2}{h} C_1 (K+1)^2 e^2} \norm{u}^2_{H_{\Phi_1}}.
\end{multline}
Here in the last inequality we have used that $\Phi_1(x) \leq C_1 \abs{x}^2$, for some $C_1\geq 1$. It follows that
\begeq
\label{eq3.2}
\norm{u}_{H_{\Phi_1}(B(0,K))}\leq {\cal O}_K(1) e^{-1/2h} \norm{u}_{H_{\Phi_1}},
\endeq
provided that the integer $N$ satisfies
\begeq
\label{eq3.3}
N \geq \frac{2C_1 (K+1)^2 e^2 +1}{h}.
\endeq
In what follows, we shall let $N_0=N_0(K)\in \nat$, $N_0 \sim h^{-1}$, be the least integer which satisfies (\ref{eq3.3}).

\bigskip
\noindent
It is now easy to derive an a priori estimate for functions in $H_{\Phi_1}(\comp^n)$, which vanish to a high order at the origin.
Let $\chi\in C^{\infty}_0(\comp^n)$, $0\leq \chi \leq 1$, be such that ${\rm supp}\,(\chi)\subset \{x\in \comp^n; \abs{x}\leq K\}$, with
$\chi(x)=1$ for $\abs{x}\leq K/2$. If $u\in H_{\Phi_1}(\comp^n)$ is such that $(1+\abs{x}^2)u\in L^2_{\Phi_1}(\comp^n)$, we have the
quantization-multiplication formula~\cite{SjDuke}, valid for $z$ in a compact subset of $\comp$,
\begin{multline*}
\left((1-\chi)\left(\widetilde{q}^w(x,hD_x)-z\right)u,u\right)_{L^2_{\Phi_1}}
\\ =\int (1-\chi(x))\left(\widetilde{q}\left(x,\frac{2}{i}\frac{\partial \Phi_1}{\partial x}(x)\right)
-z\right) \abs{u(x)}^2 e^{-2\Phi_1(x)/h}\, L(dx) + {\cal O}(h)\norm{u}^2_{H_{\Phi_1}}.
\end{multline*}
The ellipticity property,
\begeq
\label{eq3.3.1}
\Re \widetilde{q}\left(x,\frac{2}{i}\frac{\partial \Phi_1}{\partial x}(x)\right)\geq \frac{\abs{x}^2}{C_0},\quad x\in \comp^n,
\endeq
for some $C_0>1$, implies that on the support of $1-\chi$, we have,
$$
\Re\left({\widetilde{q}\left(x,\frac{2}{i}\frac{\partial \Phi}{\partial x}(x)\right)-z}\right) \geq \frac{\abs{x}^2}{2C_0},
$$
provided that $\abs{z}\leq K^2/8C_0$. Restricting the attention to this range of $z$'s and using the Cauchy-Schwarz inequality, we obtain that
\begin{multline}
\label{eq3.4}
\int (1-\chi(x)) \abs{u(x)}^2 e^{-2\Phi_1(x)/h}\, L(dx) \\
\leq {\cal O}_K(1)\norm{\left(\widetilde{q}^w(x,hD_x)-z\right)u}_{H_{\Phi_1}}\,\norm{u}_{H_{\Phi_1}} + {\cal O}_K(h)\norm{u}^2_{H_{\Phi_1}}.
\end{multline}
If $u\in H_{\Phi_1}(\comp^n)$, $(1+\abs{x}^2) u\in L^2_{\Phi_1}(\comp^n)$, is such that $\partial^{\alpha}u(0)=0$, for all $\alpha\in \nat^n$ with
$\abs{\alpha} < N_0$, then an application of (\ref{eq3.2}) shows that the left hand side of (\ref{eq3.4}) is of the form
$$
\norm{u}^2_{H_{\Phi_1}} + {\cal O}_K(h^{\infty})\norm{u}_{H_{\Phi_1}}^2.
$$
We may summarize the discussion so far in the following proposition.

\begin{prop}
Let $K>0$ be fixed and assume that $u\in H_{\Phi_1}(\comp^n)$, $(1+\abs{x}^2)u\in L^2_{\Phi_1}(\comp^n)$, is such that $\partial^{\alpha}u(0)=0$,
$\abs{\alpha} < N_0$, where $N_0\sim h^{-1}$ is the least integer such that
$$
N_0 \geq \frac{2C_1 (K+1)^2 e^2 +1}{h}.
$$
Here $\Phi_1(x) \leq C_1 \abs{x}^2$, $C_1\geq 1$. Assume also that $\abs{z}\leq K^2/8C_0$, where $C_0>1$ is the ellipticity constant in
{\rm (\ref{eq3.3.1})}. Then we have the following a priori estimate, valid for all $h>0$ sufficiently small,
$$
\norm{u}_{H_{\Phi_1}}\leq {\cal O}(1)\norm{\left(\widetilde{q}^w(x,hD_x)-z \right)u}_{H_{\Phi_1}}.
$$
\end{prop}

\bigskip
\noindent
We shall finish this section by discussing norm estimates for the linear continuous projection operator
$$
\tau_N: H_{\Phi_1}(\comp^n)\rightarrow H_{\Phi_1}(\comp^n),
$$
given by
\begeq
\label{eq3.5}
\tau_N u(x) = \sum_{\abs{\alpha}<N} (\alpha!)^{-1} \left(\partial^{\alpha}u(0)\right) x^{\alpha}.
\endeq
As in Proposition 3.2, we shall be concerned with the case when $N\in \nat$ satisfies $N\sim h^{-1}$. The projection operator
$\tau_N$ is highly non-orthogonal --- nevertheless, using the strict convexity of the quadratic weight $\Phi_1$, establishing
an exponential upper bound on its norm will be quite straightforward, as well as sufficient for our purposes. In the
following, we shall use the fact that
\begeq
\label{eq3.6}
\frac{1}{C_1} \abs{x}^2 \leq \Phi_1(x) \leq C_1 \abs{x}^2,\quad C_1\geq 1.
\endeq
Notice also that $[\tau_N,\widetilde{q}^w(x,hD_x)]=0$.

\begin{prop}
Assume that $N\in \nat$ is such that $Nh \leq {\cal O}(1)$. There exists a constant $C>0$ such that
\begeq
\label{eq3.61}
\tau_N = C e^{C/h}: H_{\Phi_1}(\comp^n)\rightarrow H_{\Phi_1}(\comp^n).
\endeq
\end{prop}
\begin{proof}
Let us observe first that when deriving the bound (\ref{eq3.61}), it suffices to restrict the attention to the
space of holomorphic polynomials, which is dense in $H_{\Phi_1}(\comp^n)$. Indeed, the analysis of~\cite{Sjostrand74} tells
us that the linear span of the generalized eigenfunctions of the quadratic operator $q^w(x,hD_x)$ is dense in $L^2(\real^n)$,
which implies the density of the holomorphic polynomials in $H_{\Phi_1}(\comp^n)$. Let
\begeq
\label{eq3.611}
u(x) = \sum_{\abs{\alpha}\leq N_1} a_{\alpha} x^{\alpha},
\endeq
for some $N_1$, where we may assume that $N_1>N$. We have
$$
\tau_N u = \sum_{\abs{\alpha}< N} a_{\alpha} x^{\alpha},
$$
and therefore, using (\ref{eq3.6}), we see that
\begeq
\label{eq3.62}
\norm{\tau_N u}^2_{H_{\Phi_1}} \leq \norm{\tau_N u}^2_{H_{\Phi_{\ell}}},
\endeq
where $\Phi_{\ell}(x) = \abs{x}^2/C_1$. When computing the expression in the right hand side of (\ref{eq3.62}), we
notice that since $\Phi_{\ell}$ is radial, we have
$$
(x^{\alpha},x^{\beta})_{H_{\Phi_{\ell}}}=0,\quad \alpha\neq \beta,
$$
while
$$
(x^{\alpha},x^{\alpha})_{H_{\Phi_{\ell}}} = \prod_{j=1}^n \int \abs{x_j}^{2\alpha_j} e^{-2\abs{x_j}^2/C_1 h}\, L(dx_j),
$$
which is immediately seen to be equal to
$$
\left(\frac{C_1 h}{2}\right)^{n+\abs{\alpha}} \pi^n \alpha!.
$$
It follows that
\begeq
\label{eq3.7}
\norm{\tau_N u}_{H_{\Phi_1}}^2 \leq \sum_{\abs{\alpha}< N} \abs{a_{\alpha}}^2 \left(\frac{C_1 h}{2}\right)^{n+\abs{\alpha}} \pi^n \alpha!.
\endeq

\medskip
\noindent
On the other hand, (\ref{eq3.6}) gives also that
\begeq
\label{eq3.8}
\norm{u}^2_{H_{\Phi_1}}\geq \norm{u}^2_{H_{\Phi_u}},
\endeq
where $\Phi_u(x) = C_1 \abs{x}^2$, and arguing as above, it is straightforward to see that the right hand side of (\ref{eq3.8}) is
given by the expression
$$
\sum_{\abs{\alpha}\leq N_1} \abs{a_{\alpha}}^2 \left(\frac{h}{2C_1}\right)^{n+\abs{\alpha}} \pi^n \alpha!.
$$
We conclude that when $u\in H_{\Phi_1}(\comp^n)$ is a holomorphic polynomial of the form (\ref{eq3.611}), then
\begeq
\label{eq3.81}
\norm{u}^2_{H_{\Phi_1}} \geq \sum_{\abs{\alpha}< N} \abs{a_{\alpha}}^2 \left(\frac{h}{2C_1}\right)^{n+\abs{\alpha}} \pi^n \alpha!.
\endeq
Combining (\ref{eq3.7}), (\ref{eq3.81}), and recalling the fact that $Nh \leq {\cal O}(1)$, we obtain the result of the proposition.
\end{proof}

\section{The finite-dimensional analysis and end of the proof}
\setcounter{equation}{0}
Let us recall the projection operator $\tau_N$, introduced in (\ref{eq3.5}). In this section, we shall analyze the resolvent of the
quadratic operator $\widetilde{q}^w(x,hD_x)$, acting on the finite-dimensional space ${\rm Im}\,\tau_N$, thereby completing the proof of
Theorem 1.1. Here $N\sim h^{-1}$. When doing so, when $m=0,1,\ldots\,$, let us introduce the finite-dimensional subspace
$E_m\subset H_{\Phi_1}(\comp^n)$, defined as the linear span of the monomials $x^{\alpha}$, with $\abs{\alpha}=m$. We have,
$$
{\rm Im}\, \tau_N  = \bigoplus_{m=0}^{N-1}E_m.
$$
We may notice here that
\begeq
\label{eq4.1}
\nu_m:={\rm dim}\, E_m  = \frac{1}{(n-1)!}(m+1)\ldots\, (m+n-1),
\endeq
and notice also that each space $E_m$ is invariant under $\widetilde{q}^w(x,hD_x)$. We shall equip ${\rm Im}\, \tau_N$ with the basis
\begeq
\label{eq4.2}
\varphi_{\alpha}(x) := \left(\pi^n\alpha!\right)^{-1/2} h^{-n/2} (h^{-1/2}x)^{\alpha},\quad \abs{\alpha}<N,
\endeq
which will be particularly convenient in the following computations, since the normalized monomials $\varphi_{\alpha}$ form an
orthonormal basis in the weighted space $H_{\Phi}(\comp^n)$, where $\Phi(x) = (1/2)\abs{x}^2$. We have,
${\rm Im}\, \tau_N \subset H_{\Phi_1}(\comp^n)\cap H_{\Phi}(\comp^n)$, in view of the strict convexity of the weights.

\bigskip
\noindent
Let us first derive an upper bound on the norm of the inverse of the operator
$$
z-\widetilde{q}^w(x,hD_x): E_m \rightarrow E_m, \quad 0\leq m < N\sim h^{-1},
$$
assuming that $E_m$ has been equipped with the $H_{\Phi}$--norm. Let $\lambda_1,\ldots,\lambda_n$ be the eigenvalues of the Hamilton map $F$ of $q$
in the upper half-plane, repeated according to their algebraic multiplicity. According to Proposition 2.1, we then have
$$
\widetilde{q}^w(x,hD_x) = \widetilde{q}^w_D(x,hD_x) + \widetilde{q}^w_N(x,hD_x),
$$
where
\begeq
\label{eq4.21}
\widetilde{q}^w_D(x,hD_x) = \sum_{j=1}^n 2\lambda_j x_j hD_{x_j} + \frac{h}{i}\sum_{j=1}^n \lambda_j,
\endeq
is the diagonal part, while
\begeq
\label{eq4.22}
\widetilde{q}^w_N(x,hD_x) = \sum_{j=1}^{n-1} \gamma_{j} x_{j+1} hD_{x_{j}},\quad \gamma_{j}\in \{0,1\},
\endeq
is the nilpotent one. It is also easily seen that the operators $\widetilde{q}^w_D(x,hD_x)$ and $\widetilde{q}^w_N(x,hD_x)$ commute.
It will be important for us to have an estimate of the order of nilpotency of the operator $\widetilde{q}^w_N(x,hD_x)$ acting on the space $E_m$.

\begin{lemma}
Let $n\geq 2$, $m\geq 1$, and let $E_m(n)$ be the space of homogeneous polynomials of degree $m$ in the variables $x_1,x_2,\ldots\, x_n$. The operator
$$
N:=\sum_{j=1}^{n-1} x_{j+1}\partial_{x_j}: E_m(n) \rightarrow E_m(n)
$$
is nilpotent of order $m(n-1)+1$.
\end{lemma}
\begin{proof}
When $\alpha=(\alpha_1,\ldots ,\alpha_n)$, $\abs{\alpha}=m$, let us write
$$
S(\alpha) = \sum_{j=1}^n j \alpha_j,
$$
and notice that $m\leq S(\alpha) \leq nm$. We have
$$
N x^{\alpha} = \sum_{\abs{\alpha'}=m,\, S(\alpha') = S(\alpha)+1} c_{\alpha'} x^{\alpha'},
$$
and similarly for powers $N^p x^{\alpha}$, where instead $S(\alpha') = S(\alpha) + p$. It follows that $N^{m(n-1)+1}x^{\alpha}$ must vanish,  as
$$
S(\alpha') = S(\alpha) + m(n-1) + 1 \geq mn+1
$$
is impossible. We also notice that $N^{m(n-1)}x_1^m =  C x_n^m\neq 0$, for some $C\neq 0$.
\end{proof}

\bigskip
\noindent
In what follows, we shall only use that the operator $\widetilde{q}^w_N(x,hD_x): E_m \rightarrow E_m$ is nilpotent of order ${\cal O}(m)$,
with the implicit constant depending on the dimension $n$ only.

\bigskip
\noindent
It is now straightforward to derive a bound on the norm of the inverse of the operator
$$
z-\widetilde{q}^w(x,hD_x): E_m \rightarrow E_m,
$$
when the space $E_m$ is equipped with the $H_{\Phi}$--norm. The matrix ${\cal D}(m)$ of the operator $\widetilde{q}^w_D(x,hD_x)$ with
respect to the basis $\varphi_{\alpha}$, $\abs{\alpha}=m$, is diagonal, with the eigenvalues of $\widetilde{q}^w(x,hD_x)$,
$$
\mu_{\alpha} = \frac{h}{i} \sum_{j=1}^n \lambda_j(2\alpha_j+1),\quad \abs{\alpha}=m,
$$
along the diagonal. On the other hand, using (\ref{eq4.2}), we compute
$$
x_{j+1}\partial_{x_j} \varphi_{\alpha} = \alpha_{j}^{1/2} (\alpha_{j+1}+1)^{1/2}
\varphi_{\alpha-e_{j}+e_{j+1}},\quad 1\leq j \leq n-1,
$$
where $\alpha=(\alpha_1,\ldots ,\alpha_n)$ and $e_{1}, \ldots ,e_{n}$ is the canonical basis in $\real^{n}$. It follows that
\begeq
\label{eq4.4}
\widetilde{q}^w_N(x,hD_x)\varphi_{\alpha} = \sum_{j=1}^{n-1} -ih \gamma_{j} \alpha_{j}^{1/2} (\alpha_{j+1}+1)^{1/2}\varphi_{\alpha-e_{j}+e_{j+1}},
\endeq
and hence the entries $({\cal N}(m)_{\alpha,\beta}) = ((\widetilde{q}^w_N(x,hD_x)\varphi_{\beta},\varphi_{\alpha}))$,
$\abs{\alpha}=\abs{\beta}=m$, of the matrix ${\cal N}(m): \comp^{\nu_m}\rightarrow \comp^{\nu_m}$ of
$\widetilde{q}^w_N(x,hD_x): E_m\rightarrow E_m$ with respect to the basis $\{\varphi_{\alpha}\}$, are bounded in modulus by
$$
h \alpha_{j}^{1/2}(\alpha_{j+1}+1)^{1/2}\leq h(m+1)\leq {\cal O}(1),
$$
since $\abs{\alpha}=m$ and $m$ does not exceed $N={\cal O}(h^{-1})$. It follows furthermore from (\ref{eq4.4}) that the matrix ${\cal N}(m)$
has no more than $n-1$ non-zero entries in any column, and a similar reasoning shows that each row of ${\cal N}(m)$ also has no more
than $n-1$ non-zero entries. Since we have just seen that the entries in ${\cal N}(m)$ are ${\cal O}(1)$, an application of
Schur's lemma shows that that the operator norm of ${\cal N}(m)$ on $\comp^{\nu_m}$ does not exceed
$$
\left(\sup_{\beta}\sum_{\alpha} \abs{{\cal N}(m)_{\alpha,\beta}}\right)^{1/2} \left(\sup_{\alpha}
\sum_{\beta} \abs{{\cal N}(m)_{\alpha,\beta}}\right)^{1/2} \leq {\cal O}(1).
$$

\bigskip
\noindent
Now the inverse of the $\nu_m\times \nu_m$ matrix
$$
z-{\cal D}(m)-{\cal N}(m): \comp^{\nu_m}\rightarrow \comp^{\nu_m},
$$
is given by
\begeq
\label{eq4.5}
\left(z-{\cal D}(m)\right)^{-1} \sum_{j=0}^{\infty} \left((z-{\cal D}(m))^{-1} {\cal N}(m)\right)^j,
\endeq
and according to Lemma 4.1 and the fact that $[\widetilde{q}^w_D(x,hD_x),\widetilde{q}^w_N(x,hD_x)]=0$, we know that the Neumann
series in (\ref{eq4.5}) is finite, containing at most ${\cal O}(m)$ terms. It follows that
\begeq
\label{eq4.6}
\left(z-{\cal D}(m)-{\cal N}(m)\right)^{-1} = \frac{\exp({\cal O}(m))}{d(z,\sigma_m)^{{\cal O}(m)}}: \comp^{\nu_m}\rightarrow \comp^{\nu_m},
\endeq
where $d(z,\sigma_m) = \inf_{\abs{\alpha}=m} \abs{z-\mu_{\alpha}}$ is the distance from $z\in \comp$ to the set of eigenvalues
$\{\mu_{\alpha}\}$ of $\widetilde{q}^w(x,hD_x)$, restricted to $E_m$.

\bigskip
\noindent
Using the fact that ${\rm Im}\, \tau_N$ is the orthogonal direct sum of the spaces $E_m$, $0\leq m \leq N-1$, we may summarize the discussion
so far in the following result.
\begin{prop}
Assume that $N\in \nat$ is such that $Nh \leq {\cal O}(1)$, and let us equip the finite-dimensional space
${\rm Im}\, \tau_N\subset H_{\Phi_1}(\comp^n)\cap H_{\Phi}(\comp^n)$ with the $H_{\Phi}$--norm, where $\Phi(x) = (1/2)\abs{x}^2$.
Assume that $z\in \comp$ satisfies ${\rm dist}\,(z,{\rm Spec}(\widetilde{q}^w(x,hD_x)))\geq h^L/C$, for some $C>0$, $L\geq 1$. Then we have
\begeq
\label{eq4.61}
\left(z-\widetilde{q}^w(x,hD_x)\right)^{-1}  = {\cal O}(1) \exp\left({\cal O}(1)h^{-1}\log \frac{1}{h}\right):
{\rm Im}\, \tau_N \rightarrow {\rm Im\, \tau}_N.
\endeq
Assuming that ${\rm dist}\,(z,{\rm Spec}(\widetilde{q}^w(x,hD_x)))\geq 1/C$, the bound {\rm (\ref{eq4.61})} improves to the following,
\begeq
\label{eq4.61.1}
\left(z-\widetilde{q}^w(x,hD_x)\right)^{-1}  = {\cal O}(1) \exp\left({\cal O}(1)h^{-1}\right): {\rm Im}\, \tau_N \rightarrow {\rm Im\, \tau}_N.
\endeq
\end{prop}

\bigskip
\noindent
{\it Remark}. Assume that the quadratic form $q$ is such that the nilpotent part in the Jordan decomposition of the Hamilton map $F$ is trivial.
The quadratic operator $\widetilde{q}^w(x,hD_x)$ acting on $H_{\Phi}(\comp^n)$ is then normal, and therefore, the estimate (\ref{eq4.61})
improves to the following,
$$
\norm{\left(z-\widetilde{q}^w(x,hD_x)\right)^{-1}}_{{\cal L}({\rm Im}\, \tau_N,{\rm Im}\, \tau_N)} \leq
\frac{1}{{\rm dist}\, (z,{\rm Spec}\, (q^w(x,hD_x)))}.
$$

\bigskip
\noindent
{\it Example}. Let $n=2$ and let us consider the semiclassical Weyl quantization of the elliptic quadratic form
$$
\widetilde{q}(x,\xi) = 2 \lambda \sum_{j=1}^2 x_j \xi_j + x_2 \xi_1, \quad \lambda = \frac{i}{2},
$$
acting on $H_{\Phi}(\comp^2)$. The eigenvalues of $\widetilde{q}^w(x,hD_x)$ are of the form $\mu_{\alpha} = h(\abs{\alpha}+1)$, $\abs{\alpha}\geq 0$,
and writing
$$
\widetilde{q}^w_D(x,hD_x) = 2\lambda\sum_{j=1}^2 x_j hD_{x_j} + \frac{2\lambda h}{i},\quad \widetilde{q}^w_N(x,hD_x) = x_2 hD_{x_1},
$$
we have
$$
\widetilde{q}^w_D(x,hD_x)\varphi_{\alpha} = \mu_{\alpha} \varphi_{\alpha},
$$
and
\begeq
\label{eq4.62}
\widetilde{q}^w_N(x,hD_x) \varphi_{\alpha} = -ih \left(\alpha_1 (\alpha_2+1)\right)^{1/2} \varphi_{\alpha-e_1+e_2}.
\endeq
Here $\varphi_{\alpha}$ have been introduced in (\ref{eq4.2}).

\medskip
\noindent
Let $\abs{\alpha}=m$, and let us write, following (\ref{eq4.5}),
\begeq
\label{eq4.7}
\left(\widetilde{q}^w(x,hD_x)-z\right)^{-1} \varphi_{\alpha} =
(\mu_{\alpha}-z)^{-1} \sum_{j=0}^m (\mu_{\alpha}-z)^{-j} \left(\widetilde{q}_N^w(x,hD_x)\right)^j \varphi_{\alpha}.
\endeq
It is then natural to take $\alpha=(m,0)$, and using (\ref{eq4.62}), a straightforward computation shows that, for $0\leq j \leq m$,
$$
\left(\widetilde{q}_N^w(x,hD_x)\right)^j \varphi_{(m,0)} = (-ih)^j \sqrt{\frac{j!m!}{(m-j)!}}\varphi_{(m-j,j)}.
$$
Let $z=1$ and take $m=h^{-1}\in \nat$ so that $\mu_{\alpha}-z=h$. By Parseval's formula,
\begeq
\label{eq4.8}
\norm{\left(\widetilde{q}^w(x,hD_x)-z\right)^{-1} \varphi_{(m,0)}}_{H_{\Phi}}^2  = \sum_{j=0}^m h^{-2} h^{-2j}h^{2j} \frac{j!m!}{(m-j)!},
\endeq
and the right hand side can be estimated from below simply by discarding all terms except when $j=m$. An application of Stirling's formula shows
that,
$$
\norm{\left(\widetilde{q}^w(x,hD_x)-z\right)^{-1} \varphi_{(m,0)}}_{H_{\Phi}} \geq m! \geq \exp\left(\frac{1}{2h}\log \frac{1}{h}\right),
$$
for all $h>0$ sufficiently small, and therefore, we see that the result of Proposition 4.2 cannot be improved. Let us finally
notice that, as can be checked directly, the quadratic operator $\widetilde{q}^w(x,hD_x)$ acting on $H_{\Phi}(\comp^2)$ is
unitarily equivalent, via an FBI-Bargmann transform, to the quadratic operator
$$
q(x,hD_x): L^2(\real^n)\rightarrow L^2(\real^n),
$$
of the form
$$
q(x,hD_x) = q_0(x,hD_x) - \frac{i}{2} a_2^* a_1,
$$
where
$$
q_0(x,hD_x) = -\frac{h^2}{2}\Delta + \frac{x^2}{2} = \frac{1}{2}\left(a_1^* a_1 + a_2^* a_2\right) + h,
$$
is the semiclassical harmonic oscillator, while
$$
a_j^*  = x_j - h\partial_{x_j},\quad a_j = x_j + h\partial_{x_j},\quad j=1,2,
$$
are the creation and annihilation operators, respectively. See also~\cite{CGrSj}.

\bigskip
\noindent
We shall now complete the proof of Theorem 1.1 in a straightforward manner, combining our earlier computations and estimates.
Elementary considerations, analogous to those used in the proof of Proposition 3.3, show that for some constant $C>0$, we have,
when $u\in {\rm Im}\, \tau_N$,
\begeq
\label{eq4.9}
\norm{u}_{H_{\Phi_1}}\leq C e^{C/h} \norm{u}_{H_{\Phi}},\quad \norm{u}_{H_{\Phi}}\leq C e^{C/h} \norm{u}_{H_{\Phi_1}}.
\endeq
It follows therefore that the result of Proposition 4.2,
\begeq
\label{eq4.10}
\left(z-\widetilde{q}^w(x,hD_x)\right)^{-1}  = {\cal O}(1) \exp\left({\cal O}(1)h^{-1}\log \frac{1}{h}\right): {\rm Im}\,
\tau_N \rightarrow {\rm Im\, \tau}_N,
\endeq
holds also when the space ${\rm Im}\, \tau_N\subset H_{\Phi_1}(\comp^n)\cap H_{\Phi}(\comp^n)$ is equipped with the $H_{\Phi_1}$-norm,
at the expense of an ${\cal O}(1)$--loss in the exponent. The same conclusion holds for the bound (\ref{eq4.61.1}).

\bigskip
\noindent
Let $\Omega \subset \subset \comp$ and assume that $z\in \Omega \subset \subset \comp$ is such that
${\rm dist}\, (z,{\rm Spec}(\widetilde{q}^w(x,hD_x)))\geq h^L/C$, for some $L\geq 1$ and $C>0$ fixed. Then according to Proposition 3.2,
there exists $N_0 \in \nat$, $N_0\sim h^{-1}$, such that if $u\in H_{\Phi_1}(\comp^n)$, is such that
$(1+\abs{x}^2)u\in L^2_{\Phi_1}(\comp^n)$, then, using that $[\widetilde{q}^w(x,hD_x),\tau_{N_0}]=0$, we get, for all $h>0$ small enough,
\begin{multline}
\label{eq4.11}
\norm{\left(1-\tau_{N_0}\right)u}_{H_{\Phi_1}} \leq {\cal O}(1) \norm{\left(\widetilde{q}^w(x,hD_x)-z\right)\left(1-\tau_{N_0}\right)u}_{H_{\Phi_1}}
\\ \leq {\cal O}(1) \exp \left({\cal O}(1) h^{-1}\right)\norm{\left(\widetilde{q}^w(x,hD_x)-z\right)u}_{H_{\Phi_1}}.
\end{multline}
Here we have also used Proposition 3.3. On the other hand, the bound (\ref{eq4.10}) and Proposition 3.3 show that
\begin{multline}
\label{eq4.12}
\norm{\tau_{N_0}u}_{H_{\Phi_1}} \leq {\cal O}(1) \exp\left({\cal O}(1)h^{-1}\log \frac{1}{h}\right)
\norm{\tau_{N_0}\left(\widetilde{q}^w(x,hD_x)-z\right)u}_{H_{\Phi_1}}\\
\leq {\cal O}(1)\exp\left({\cal O}(1)h^{-1}\log \frac{1}{h}\right) \norm{\left(\widetilde{q}^w(x,hD_x)-z\right)u}_{H_{\Phi_1}}.
\end{multline}
Combining (\ref{eq4.11}) and (\ref{eq4.12}), we obtain the bound (\ref{est2}). The estimate (\ref{est1}) follows in a similar way, and
hence, the proof of Theorem 1.1 is complete.


\begin{thebibliography}{40}

\bibitem{Boulton} L. Boulton, {\it Non-self-adjoint harmonic oscillator, compact semigroups and pseudospectra},
J. Operator Theory {\bf 47} (2002), 413-–429.

\bibitem{CGrSj} E. Caliceti, S. Graffi, and J. Sj\"ostrand, {\it ${\cal PT}$--symmetric non-self-adjoint operators, diagonalizable and non-diagonalizable, with a real discrete spectrum}, J. Phys. A {\bf 40} (2007), 10155-–10170.


\bibitem{Da2} E. B. Davies, {\it Wild spectral behaviour of anharmonic oscillators},
Bull. London Math. Soc. {\bf 32} (2000), 432-–438.

\bibitem{DaKu} E. B. Davies and A. B. J. Kuijlaars, {\it Spectral asymptotics of the non-self-adjoint harmonic oscillator},
J. London Math. Soc. {\bf 70} (2004), 420-–426.

\bibitem{DeSjZw} N. Dencker, J. Sj\"ostrand, and M. Zworski, {\it Pseudo-\-spectra of se\-mi\-clas\-si\-cal
(pseudo)\-diffe\-ren\-tial ope\-rators}, Comm. Pure Appl. Math. {\bf 57} (2004), 384-415.


\bibitem{EmTr} M. Embree and L. N. Trefethen, {\it Spectra and pseudospectra. The behavior of nonnormal matrices and operators},
Princeton University Press, Princeton, NJ, 2005.

\bibitem{GeSj} C. G\'erard and J. Sj\"ostrand, {\it Semiclassical resonances generated by a closed trajectory of hyperbolic type}, Comm. Math.
Phys. {\bf 108} (1987), 391--421.



\bibitem{HeSj84} B.~Helffer and J.~Sj\"ostrand, {\it Multiple wells in the semiclassical limit. {\rm I}.},
Comm. P.D.E. {\bf 9} (1984), 337-–408.



\bibitem{Hi2004} M. Hitrik, {\it Boundary spectral behavior for semiclassical operators in dimension one}, International
Mathematics Research Notices {\bf 64} (2004), 3417--3438.





\bibitem{H_book}
L. H\"{o}rmander, {\it The analysis of linear partial differential operators} (vol. I--IV), Springer Verlag (1985).


\bibitem{Markus} A. S. Markus, {\it Introduction to the spectral theory of polynomial operator pencils}, Translations
of Mathematical Monographs 71, American Mathematical Society, Providence RI, 1998.


\bibitem{Pravda_XEDP} K. Pravda-Starov, {\it Sur le pseudo-spectre de certaines classes d'op\'erateurs pseudo-diff\'erentiels non auto-adjoints},
S\'eminaire \'Equations aux D\'eriv\'ees Partielles. 2006–2007, Exp. No. XV, 35 pp., \'Ecole Polytech., Palaiseau, 2007.

\bibitem{PravdaDuke} K. Pravda-Starov, {\it On the pseudospectrum of elliptic quadratic differential operators},
Duke Math. J. {\bf 145} (2008), 249-–279.

\bibitem{Sjostrand74}
J. Sj\"{o}strand, {\it Parametrices for pseudodifferential operators with multiple characteristics},
Ark. f\"{o}r Matematik {\bf 12} (1974), 85-130.

\bibitem{Sj82} J. Sj\"ostrand, {\it Singularit\'es analytiques microlocales}, Ast\'erisque, {\bf 95} (1982), 1-166,
Soc. Math. France, Paris.

\bibitem{Sj86} J. Sj\"ostrand, {\it Semiclassical resonances generated by nondegenerate critical points},
Pseudodifferential operators (Oberwolfach, 1986), 402–-429, Lecture Notes in Math., {\bf 1256}, Springer, Berlin, 1987


\bibitem{SjDuke} J. Sj\"ostrand, {\it Geometric bounds on the density of resonances for semiclassical problems}, Duke Math. J. {\bf 60} (1990), 1--57.

\bibitem{Sj95} J. Sj\"ostrand, {\it Function spaces associated to
global I-Lagrangian manifolds}, Structure of solutions of differential
equations, Katata/Kyoto, 1995, World Sci. Publ., River Edge, NJ (1996).

\bibitem{Sj2010} J. Sj\"ostrand, {\it Resolvent estimates for non-selfadjoint operators via semigroups},
Around the research of Vladimir Maz'ya. III, 359–384, Int. Math. Ser. (N. Y.), 13, Springer, New York, 2010.

\bibitem{Viola} J. Viola, {\it Resolvent estimates for non-selfadjoint operators with double characteristics}, J. London Math. Soc., to appear.

\end{thebibliography}
\end{document}